\documentclass[12pt]{amsart}
\usepackage{amssymb,amsmath}
\def\dist {{\rm dist}}
\def\Frac {{\rm dim}}
\def\dom {{\rm dom}}
\def\x{\boldsymbol{x}}
\def\z{\boldsymbol{z}}
\def\eti{\boldsymbol{\eta}}
\def\xei{\boldsymbol{\xi}}
\def\A {\boldsymbol{A}}
\def\B {\boldsymbol{B}}
\def\eps {\varepsilon}
\def\d {{\rm d}}
\def\e {{\rm e}}
\def\I {{\mathfrak I}}
\def\R {\mathbb{R}}
\def\Q {Q}
\def\q {q}
\def\W {{\mathcal W}}

\def\H {{\mathcal H}}
\def\V {{\mathcal V}}
\def\Bcal {{\mathcal B}}
\def\BB {{{\mathbb B}}}
\def\C {{\mathcal C}}
\def\D {{\mathfrak D}}
\def\E {{\mathfrak E}}
\def\K {{\mathcal K}}

\def\T {{\mathbb T}}
\def\LL {{\boldsymbol{\Lambda}}}

\def\S {{\mathcal S}}
\def\M {{\mathcal M}}
\def \l {\langle}
\def \r {\rangle}
\def \pt {\partial_t}
\def \ptt {\partial_{tt}}
\def \HB {\mathbb{H}}
\def \HH {{\rm H}}
\def\ddt{\frac{\d}{\d t}}

\newtheorem{proposition}{Proposition}[section]
\newtheorem{theorem}[proposition]{Theorem}

\newtheorem{lemma}[proposition]{Lemma}
\theoremstyle{definition}

\newtheorem{assumption}[proposition]{Condition}
\newtheorem{remark}[proposition]{Remark}

\numberwithin{equation}{section}

\def \au {\rm}
\def \ti {\it}
\def \jou {\rm}
\def \bk {\it}
\def \no#1#2#3 {{\bf #1} (#3), #2.}
\def \eds#1#2#3 {#1, #2, #3.}

\title[Exponential attractors for equations with memory]
{Exponential attractors for abstract equations\\ with memory
and applications to viscoelasticity}

\author[V. Danese, P.G. Geredeli and V. Pata]
{Valeria Danese, Pelin G. Geredeli and Vittorino Pata}

\address{Politecnico di Milano - Dipartimento di Matematica ``F.\ Brioschi''
\newline\indent
Via Bonardi 9, 20133 Milano - Italy}
\email{valeria.danese@polimi.it {\rm (V. Danese)}}
\email{vittorino.pata@polimi.it {\rm (V. Pata)}}

\address{Department of Mathematics, Faculty of Science
\newline\indent
Hacettepe University, Beytepe 06800, Ankara - Turkey}
\email{pguven@hacettepe.edu.tr {\rm (P.G. Geredeli)}}

\thanks{P.G.G.\ was partially supported by the {\it Scientific Research Projects Coordination Unit of Hacettepe University, Ankara.}}

\begin{document}

\begin{abstract}
We consider an abstract equation with memory of the form
$$\pt \x(t)+\int_{0}^\infty k(s) \A\x(t-s)\d s+\B\x(t)=0$$
where $\A,\B$ are operators acting on some Banach space, and the convolution kernel $k$ is a
nonnegative convex summable function of unit mass.
The system is translated into an ordinary differential equation on a Banach space accounting for the presence of memory,
both in the so-called history space framework and in the minimal state one.
The main theoretical result is a theorem providing sufficient conditions in order for the related solution semigroups to possess
finite-dimensional exponential attractors.
As an application, we prove the existence of exponential attractors for the integrodifferential equation
$$\ptt u - h(0)\Delta u - \int_{0}^\infty h'(s) \Delta u(t-s)\d s+ f(u) = g$$
arising in the theory of isothermal viscoelasticity, which is just a particular concrete realization
of the abstract model, having defined the new kernel
$h(s)=k(s)+1$.
\end{abstract}

\maketitle

\section{Introduction}
\label{Sec1}

\noindent
A large class of physical phenomena in which
delay effects occur, such as
viscoelasticity, population dynamics or heat flow in real conductors,
are modeled by equations with memory, where the dynamics is influenced
by the past history of the variables in play through a convolution integral.
Given a real Banach space $X$,
the general structure of an equation with memory
in the unknown $\x=\x(t):\R\to X$ reads as follows:
\begin{equation}
\label{BASEL}
\pt \x(t)
+\int_{0}^\infty k(s) \A\x(t-s)\d s+\B\x(t)=0.
\end{equation}
Here, the convolution (or memory) kernel $k$ is a nonnegative
summable function
of total mass
$$
\int_0^\infty k(s)\d s=1,
$$
having the explicit form
$$
k(s)=\int_s^\infty \mu(y)\d y,
$$
where $\mu\in L^1(\R^+)$ is a nonincreasing
(hence nonnegative) piecewise absolutely continuous function, possibly
unbounded about zero.
The discontinuity points of $\mu$, if any,
form an increasing sequence $s_n$, which can be either finite or $s_n\to\infty$.
Assuming without loss of generality $\mu$ right-continuous,
we denote by
$$\mu_n=\mu(s_n^-)-\mu(s_n)>0$$
the jump amplitudes at the (left) discontinuity points $s_n$.
Besides, $\mu$ is supposed to satisfy for every $t,s>0$ and
some $\Theta\geq 1$ and $\delta>0$ the inequality
\begin{equation}
\label{NEC}
\mu(t+s)\leq \Theta \e^{-\delta t}\mu(s).
\end{equation}
Concerning $\A$ and $\B$, they are (possibly nonlinear) operators
densely defined on $X$. For any fixed time $t$,
the operator $\B$ acts on $\x(t)$, the instantaneous value of $\x$,
while
$\A$ acts on the past history of $\x$, namely,
the past values of $\x$ up to $t$.
The variable $\x(t)$ is supposed to solve the equation in some weak sense for every
$t>0$, whereas it is regarded as a known initial datum for negative times.
Accordingly, \eqref{BASEL} is supplemented with
the ``initial condition"
\begin{equation}
\label{BASEID}
\x(-s)={\boldsymbol{h}}(s),\quad\forall s\geq 0,
\end{equation}
where ${\boldsymbol{h}}:[0,\infty)\to X$
is a given function accounting for the {\it initial} past history of $\x$.

A common feature of equations arising from concrete physical models is the
presence of dissipation mechanisms. In this perspective,
here we are mostly interested to the longterm behavior
of solutions. As we will show in a while, it is possible
to translate the original problem within a semigroup framework. Accordingly,
the dissipativity properties of~\eqref{BASEL} can be described in terms of
``small" sets of the phase space able to eventually capture the trajectories
of the related solution semigroup $S(t)$ acting on a suitable Banach space $\H$.
Dealing with semigroups, an important object is
the {\it global attractor} (see e.g.\ \cite{BV,CV,HAL,HAR,ROB,TEM}), whose existence has been
proved for several models with memory. Instead, our main goal is to discuss the existence
of {\it exponential attractors}, otherwise called {\it inertial sets},
firstly introduced in~\cite{EFNT} in a Hilbert space setting (see~\cite{EMZ} for the Banach space case),
which have the advantage of being more stable than global attractors,
and attract trajectories exponentially fast (see~\cite{ABS,MZ} for a detailed discussion).
By definition, an exponential attractor for a semigroup $S(t)$ acting on a Banach space $\H$
is a compact set $\E\subset\H$ satisfying the following properties:
\begin{itemize}
\item[$\bullet$] $\E$ is positively invariant for the semigroup, namely,
$S(t)\E\subset\E$ for all $t\geq 0$.
\smallskip
\item[$\bullet$] $\E$ has finite fractal dimension in $\H$.
\smallskip
\item[$\bullet$] $\E$ is exponentially attracting for $S(t)$, i.e.\ there exist
an exponential rate $\omega>0$
and a nondecreasing positive function $\Q$ such that
$$\dist_\H(S(t)\Bcal,\E)\leq \Q(\|\Bcal\|_\H)\e^{-\omega t}
$$
for every bounded subset $\Bcal\subset\H$.
\end{itemize}

\smallskip
Recall that the fractal dimension of $\E$ in $\H$ is defined as
$$\Frac_{\H}(\E)=\limsup_{r\to 0}
\frac{ \ln N(r)}{\ln \frac1r},$$
where $N(r)$ is
the smallest number of $r$-balls of $\H$ necessary to cover $\E$, whereas
$$\dist_\H(\Bcal_1,\Bcal_2)=\sup_{b_1\in\Bcal_1}\inf_{b_2\in \Bcal_2}
\|b_1-b_2\|_\H$$
denotes the standard Hausdorff semidistance in $\H$ between two
sets $\Bcal_1$ and $\Bcal_2$.

\smallskip
Indeed, although the existence of global attractors for equations with memory has been investigated
in several papers, there are considerably fewer results concerning exponential attractors.
This is mainly due to the technical difficulties arising in the application of
the classical techniques to this particular framework.

\subsection*{Outline of the paper} As a first step,
we show how to translate the original
integro\-\-differential problem~\eqref{BASEL}
into a dynamical system within two different frameworks,
the so-called {\it past history} framework and the {\it minimal state} one.
This is done in a heuristic way
in the next Sec.\ \ref{SecTrEq}, and it is then formalized in the proper mathematical setting
in Sec.\ \ref{SECSOL}, upon defining suitable functional spaces (see
Sec.\ \ref{SECGEN}). This leads to the generation of two strongly continuous semigroups $S(t)$
(on the past history space) and $\hat S(t)$ (on the minimal state space)
describing the solutions to~\eqref{BASEL}. Sec.\ \ref{SECEXPH} is devoted to the main
abstract result of this work: namely, we provide sufficient conditions in order for the semigroup $S(t)$
to possess a regular exponential attractor. As a corollary, in Sec.\ \ref{SECEXPS} the analogous
result is shown to hold for the semigroup $\hat S(t)$.
In the second part of the paper,
we discuss an application of the abstract theorems to the well-known equation
of viscoelasticity (see Sec.\ \ref{SecVISCO} and Sec.\ \ref{SecLAST}).

\section{The Transformed Equation}
\label{SecTrEq}

\noindent
In order to view~\eqref{BASEL} as a dynamical system,
two different strategies have been devised.

\subsection{The past history framework}
Using a method proposed by Dafermos \cite{DAF} (see also \cite{FAR}),
we introduce for $(t,s)\in[0,\infty)\times\R^+$
the {\it summed past history}
$$
\eti^t(s)=\int_0^s \A\x(t-y)\d y,
$$
which (formally) fulfills the differential equation of hyperbolic type
$$\pt \eti^t(s)=-\partial_s\eti^t(s) +\A\x(t),$$
subject to the ``boundary condition"
$$\lim_{s\to 0}\eti^t(s)=0.$$
A formal integration by parts yields the equality
$$
\int_0^\infty k(s)\A\x(t-s)\d s
=\int_0^\infty \mu(s)\eti^t(s)\d s.
$$
Hence, the original equation \eqref{BASEL} translates into the evolution system
in the unknown variables $\x=\x(t)$ and $\eti=\eti^t(s)$
\begin{equation}
\label{SYSM}
\begin{cases}
\pt \x(t)+\displaystyle\int_0^\infty \mu(s) \eti^t(s)\d s+\B\x(t)=0,\\
\noalign{\vskip1mm}
\pt\eti^t(s)=-\partial_s\eti^t(s)+\A\x(t).
\end{cases}
\end{equation}
In turn, the initial condition \eqref{BASEID} becomes
$$
\begin{cases}
\x(0)=\x_0,\\
\eti^0=\eti_0,
\end{cases}
$$
where we set
$$
\x_0={\boldsymbol{h}}(0)
\qquad\text{and}\qquad
\eti_0(s)=\int_0^s \A {\boldsymbol{h}}(y)\d y.
$$

\begin{remark}
If $\A$ is linear, and so it commutes with the integral,
one usually defines $\eti$ in a slightly different way,
namely,
$\eti^t(s)=\int_0^s \x(t-y)\d y$.
Clearly, \eqref{SYSM} changes accordingly.
\end{remark}

\subsection{The minimal state framework}
Although it successfully allows to view the original integrodifferential equation
as an abstract differential equation, the past history framework suffers from a structural theoretical drawback. Indeed,
it is clear that the dynamics is known once $\x_0$ and $\eti_0$ are assigned. This means that
the initial past history $\boldsymbol{h}$ may not be recovered from the future evolutions,
hence it has no possibilities to be an observable quantity.
To cope with this fact, an alternative approach has been proposed
(see~\cite{DD,DFG,FGP}), based on the notion of {\it minimal state}: an additional
variable accounting for the past history
which contains the necessary and sufficient information determining the
future dynamics. Precisely, for $(t,\tau)\in[0,\infty)\times\R^+$, we define
$$\xei^t(\tau)=\int_0^\infty\mu(\tau+s)\A\x(t-s)\d s,$$
which (again, formally) satisfies the relations
$$
\pt\xei^t(\tau)=\partial_\tau \xei^t(\tau)+\mu(\tau)\A\x(t)
$$
and
$$
\int_0^\infty \xei^t(\tau)\d \tau=\int_0^\infty k(s)\A\x(t-s)\d s.
$$
Hence, \eqref{BASEL} takes the form
\begin{equation}
\label{SYSS}
\begin{cases}
\displaystyle
\pt \x(t)+\displaystyle\int_0^\infty \xei^t(\tau)\d \tau+\B\x(t)=0,\\
\noalign{\vskip1mm}
\pt \xei^t(\tau)=\partial_\tau \xei^t(\tau)+\mu(\tau) \A\x(t),
\end{cases}
\end{equation}
and the initial condition \eqref{BASEID}
translates into
$$
\begin{cases}
\x(0)=\x_0,\\
\xei^0=\xei_0,
\end{cases}
$$
where we set
$$
\x_0={\boldsymbol{h}}(0)
\qquad\text{and}\qquad
\xei_0(\tau)=\int_0^\infty\mu(\tau+s)\A {\boldsymbol{h}}(s)\d s.
$$

\begin{remark}
Such a description meets the sought minimality requirement.
Indeed, once the initial state $\xei_0$ is assigned, we can write
$$
\xei^t(\tau)=\xei_0(t+\tau)+\int_0^t\mu(\tau+s)\A\x(t-s)\d s.
$$
Then, plugging the latter equality into the first equation of~\eqref{SYSS},
we deduce for every $t\geq 0$
the relation
$$\int_t^\infty\xei_0(\tau)\d\tau=G(t),
$$
for some function $G$ depending only on the values of $\x(t)$ for $t$ positive.
Hence, the
knowledge of $\x(t)$ for all $t\geq 0$ uniquely determines
$\xei_0$.
\end{remark}

\begin{remark}
Similarly to the past history case, if $\A$ is linear it is customary
to define instead
$\xei^t(\tau)=\int_0^\infty\mu(\tau+s)\x(t-s)\d s$,
and change~\eqref{SYSS} accordingly.
\end{remark}

\section{Functional Setting and Notation}
\label{SECGEN}

\noindent
We first introduce the abstract functional setting
needed to carry out our analysis.

\subsection{Notation}
We denote by $\I$ the set of nondecreasing
functions $\Q:[0,\infty)\to [0,\infty)$, and by $\D$ the set of nonincreasing
functions $\q:[0,\infty)\to [0,\infty)$ vanishing at infinity.
Besides, given a Banach space $\H$ and $r\geq 0$, we set
$$\BB_\H(r)=\{z\in \H:\, \|z\|_\H\leq r\}.$$

\subsection{Geometric spaces}
Let $X^0,X^1$ and $Y^0,Y^1$ be reflexive Banach spaces having
dense and compact embeddings
$$X^1\Subset X^0
\qquad\text{and}\qquad Y^1\Subset Y^0.
$$

\subsection{Memory spaces}
For $\imath=0,1$, we introduce
the {\it memory spaces}
$$
\M^\imath=L^2_{\mu}(\R^+;Y^\imath),
$$
namely, the spaces of $L^2$-functions on $\R^+$ with values in $Y^\imath$
with respect to the measure $\mu(s)\d s$,
normed by
$$
\|\eti\|^2_{\M^\imath}
=\int_0^\infty\mu(s)
\|\eti(s)\|^2_{Y^\imath}\d s.
$$
If $Y^\imath$ is a Hilbert space, so is $\M^\imath$.
Albeit clearly continuous, the
embedding $\M^1\subset \M^0$ is not compact (see \cite{PZ} for a counterexample).
We will also consider the linear operator
$$T\eti=-\eti',
\qquad
\dom(T)=\big\{\eti\in\M^0:\,\eti'\in\M^{0},\,\,\eti(0)=0\big\},
$$
where the {\it prime} stands for weak derivative and $\eti(0)=\lim_{s\to 0}\eti(s)$ in $Y^0$.
It can be verified that $T$ is
the infinitesimal generator of the right-translation
semigroup $R(t)$ on $\M^0$, defined as
$$
(R(t)\eti)(s)=
\begin{cases}
0 & s\leq t,\\
\eti(s-t) & s>t.
\end{cases}
$$
Finally, for $\imath=0,1$, we define the {\it extended memory spaces}
$$\H^\imath= X^\imath\times\M^\imath,
$$
which are Banach spaces with respect to the norms
$$\|(\x,\eti)\|_{\H^\imath}^2=\|\x\|_{X^\imath}^2
+\|\eti\|_{\M^\imath}^2,$$
and we call $\Pi_1$ and $\Pi_2$
the projections of $\H^0$
onto its components $X^0$ and $\M^0$,
namely,
$$\Pi_1 (\x,\eti)=\x,\qquad \Pi_2 (\x,\eti)=\eti.$$

\subsection{State spaces}
We define the nonnegative kernel $\nu$ as
$$
\nu(\tau)
= 1/\mu(\tau),
$$
where $\nu(0)=\lim_{\tau\to 0}1/\mu(\tau)$,
and we agree to set $\nu(\tau)=0$ whenever $\mu(\tau)=0$, in order to include
the finite delay case in our discussion.
The assumptions on $\mu$ imply that $\nu$ is nondecreasing and
piecewise absolutely continuous.
In particular, \eqref{NEC} yields
\begin{equation}
\label{NECnu}
\nu(\tau-s)\leq \Theta \e^{-\delta s}\nu(\tau), \quad\forall s<\tau,
\end{equation}
provided that $\nu(\tau)>0$.
Then, for $\imath=0,1$, we introduce the {\it (minimal) state spaces}
$$\S^{\imath}=L^2_\nu(\R^+;Y^\imath)$$
with norms
$$\|\xei\|^2_{\S^{\imath}}
=\int_0^{\infty}\nu(\tau)\|\xei(\tau)\|^2_{Y^\imath}\d\tau,$$
along with the {\it extended state spaces}
$$
\V^\imath = X^\imath\times \S^\imath
$$
normed by
$$\|(\x,\xei)\|_{\V^\imath}^2=\|\x\|_{X^\imath}^2
+\|\xei\|_{\S^\imath}^2.$$
Besides, let $P$ be the infinitesimal generator
of the left-translation semigroup on $\S^0$,
namely,
$$P\xei=\xei',\qquad
\dom(P)=\big\{\xei\in\S^0:\,\xei'\in\S^0\big\}.$$
Abusing the notation, we keep denoting by $\Pi_1$ and $\Pi_2$
the projections of $\V^0$
onto its components $X^0$ and $\S^0$, respectively.

\section{The Solution Semigroups}
\label{SECSOL}

\noindent
We are now in the position to rephrase \eqref{SYSM} and \eqref{SYSS}
within the correct functional setting.

\subsection{The past history framework}
We interpret \eqref{SYSM} as the ODE in $\H^0$
in the unknown variables $\x=\x(t)$ and $\eti=\eti^t(s)$
\begin{equation}
\label{SYM}
\begin{cases}
\pt \x(t)+\displaystyle\int_0^\infty \mu(s) \eti^t(s)\d s+\B\x(t)=0,\\
\noalign{\vskip1mm}
\pt\eti^t=T\eti^t+\A\x(t).
\end{cases}
\end{equation}
Throughout this work,
we assume that \eqref{SYM} has a unique global solution in some weak sense
for every initial datum $\z=(\x_0,\eti_0)\in\H^0$.
This is the same as saying that~\eqref{SYM} generates a semigroup
$$S(t):\H^0\to\H^0.$$
We do not require any continuity in time, although in most concrete cases one has
$$t\mapsto S(t)\z\in\C([0,\infty),\H^0),\quad\forall\z\in\H^0.$$
In particular, once $\x(t)$ is known, the second component $\eti^t=\Pi_2 S(t)\z$ of the solution
is recovered as a Duhamel integral, hence it fulfills the explicit representation formula
\begin{equation}
\label{REPETA}
\eti^t(s)
=(R(t)\eti_0)(s)+\int_0^{\min\{t,s\}}\A\x(t-y)\d y.
\end{equation}
Besides, $S(t)$ is supposed to satisfy the following H\"older continuity property:
there exists $\Q\in\I$ and $\kappa=\kappa(r)\in(0,1]$ such that
\begin{equation}
\label{contM}
\|S(t)\z_1-S(t)\z_2\|_{\H^0}\leq \Q(r)\e^{\Q(r)t}\|\z_1-\z_2\|_{\H^0}^\kappa,
\end{equation}
whenever $\z_1,\z_2\in \BB_{\H^0}(r)$.

\subsection{The minimal state framework}
In a similar manner, we interpret \eqref{SYSS} as the ODE in $\V^0$
in the unknown variables $\x=\x(t)$ and $\xei=\xei^t(\tau)$
\begin{equation}
\label{SYS}
\begin{cases}
\pt \x(t)+\displaystyle\int_0^\infty \xei^t(\tau)\d \tau+\B\x(t)=0,\\
\noalign{\vskip1mm}
\pt \xei^t=P \xei^t+\mu \A\x(t).
\end{cases}
\end{equation}
Again, for every initial datum $\hat\z=(\x_0,\xei_0)\in\V^0$,
we assume that \eqref{SYS} generates a semigroup
$$\hat S(t):\V^0\to\V^0.$$
Hence, arguing as in the previous case, the second component $\xei^t=\Pi_2\hat S(t)\hat \z$ of the solution fulfills
the explicit representation formula
\begin{equation}
\label{REP}
\xei^t(\tau)=\xei_0(t+\tau)+\int_0^t\mu(\tau+s)\A\x(t-s)\d s.
\end{equation}
Finally, we suppose that
there exists $\Q\in\I$ and $\kappa=\kappa(r)\in(0,1]$ such that
\begin{equation}
\label{contS}
\|\hat S(t)\hat\z_1-\hat S(t)\hat\z_2\|_{\V^0}\leq \Q(r)\e^{\Q(r)t}\|\hat\z_1-\hat\z_2\|_{\V^0}^\kappa,
\end{equation}
whenever $\hat\z_1,\hat\z_2\in \BB_{\V^0}(r)$.

\subsection{The map $\LL$}
The connection between the memory and the state spaces
has been devised in \cite{comarpa,FGP}.
Setting for any $\eti\in\M^0$
$$
(\Lambda \eti) (\tau)=-\int_0^\infty\mu'(\tau+s)\eti(s)\d s
+\sum_{\tau<s_n}\mu_n\eti(s_n-\tau),
$$
the following lemma is proved.

\begin{lemma}
\label{TTTH}
The map $\LL$ defined by
$$(\x,\eti)\mapsto \LL (\x,\eti)=(\x,\Lambda\eti)$$
is a bounded linear operator of unitary norm from $\H^\imath$ into $\V^\imath$.
Moreover, for every $\eti\in\M^0$ and every $\tau>0$, we have the equality
\begin{equation}
\label{qqqqqqqq}
\int_0^\infty\mu(\tau+s)\eti(s)\d s=\int_{\tau}^\infty (\Lambda\eti)(y)\d y=(\Lambda\boldsymbol{H})(\tau),
\end{equation}
where $\boldsymbol{H}(s)=\int_0^s\eti(y)\d y$.
\end{lemma}

\begin{remark}
However, it should be observed that, in general, the map $\LL$ is neither 1-to-1 nor onto.
\end{remark}

The link between the two formulations
is detailed in the next lemma (cf.\ \cite{comarpa,FGP}),
showing in particular that the state approach describes the dynamics
in greater generality.

\begin{lemma}
\label{lemme}
For every $\z\in\H^0$,
the following equality holds:
$$\hat S(t) \LL \z=\LL S(t)\z.$$
\end{lemma}

\section{Exponential Attractors: The Past History Framework}
\label{SECEXPH}

\subsection{Statement of the theorem}
Our main goal is to provide a ``user friendly" recipe establishing sufficient conditions
in order for the semigroup $S(t)$ acting on $\H^0$ to possess
a regular exponential attractor $\E$. Such a result is specifically tailored
for equations with memory: as it will be shown in the final application,
the abstract hypotheses are easily translated in concrete terms.
A similar attempt has been made in~\cite{Relax}, but only for a particular case of our general equation~\eqref{BASEL}.
In principle, this theorem can be inferred from the previous work~\cite{GMPZ},
dealing with the existence of a H\"older continuous family of exponential attractors $\E_\eps$
for a family of semigroups $S_\eps(t)$
depending on a perturbation parameter $\eps\in [0,1]$; see \cite[Theorem~4.4]{GMPZ}.
Nonetheless, for the basic (but still very interesting) case of a single semigroup
of memory type, it is not at all immediate how to derive a clean and simple
statement of wide application from~\cite{GMPZ}. Accordingly, we will give the proof in some detail.

\begin{theorem}
\label{THM-ExpM}
In addition to our general assumptions, suppose that the following hold:
\begin{enumerate}
\item[{\bf (i)}] There exists $R_1>0$ such that, given any $r\geq 0$,
$$
\| S(t)\z\|_{\H^1} \leq \q_r(t) + R_1
$$
for some $\q_r \in \D$ and every $\z \in \BB_{\H^1}(r)$.
\smallskip
\item[{\bf (ii)}] There is $r_1>0$ such that the ball
$\BB_{\H^1}(r_1)$
is exponentially attracting for $S(t)$.
\smallskip
\item[{\bf (iii)}] The operator $\boldsymbol A$ maps bounded subsets of $X^1$ into bounded subsets of $Y^0$.
\smallskip
\item[{\bf (iv)}] There exist $p\in(1,\infty]$ and $\Q \in \I$ such that,
for every $r \geq 0$ and every $\theta >0$ sufficiently large,
$$
\| \pt \x\|_{L^p(\theta,2\theta;X^0)}\leq \Q(r + \theta)
$$
for all $\x(t)=\Pi_1 S(t)\z$ with $\z \in \BB_{\H^1}(r)$.
\smallskip
\item[{\bf (v)}] For every $r>0$ there are $\Q_r \in \I$ and
$\q_r \in \D$ such that, for all $\z_1, \z_2 \in \BB_{\H^1}(r)$,
$$
S(t)\z_1- S(t)\z_2 = L(t, \z_1, \z_2) + K(t, \z_1,\z_2),
$$
where the maps $L$ and $K$ satisfy
\begin{align*}
\| L(t, \z_1, \z_2)\|_{\H^0} &\leq \q_r(t) \| \z_1 - \z_2\|_{\H^0}, \\
\| K(t, \z_1, \z_2) \|_{\H^1} &\leq \Q_r(t)\| \z_1 - \z_2\|_{\H^0}.
\end{align*}
Moreover, let $\bar\eti^t=\Pi_2 K(t,\z_1, \z_2)$
fulfill the Cauchy problem
$$
\begin{cases}
\partial_t \bar\eti^t = T\bar\eti^t + \boldsymbol{w}(t),\\
\bar\eti^0 = 0,
\end{cases}
$$
for some function $\boldsymbol{w}$ satisfying the estimate
$$
\|\boldsymbol{w}(t)\|_{Y^0} \leq \Q_r(t)\| \z_1 - \z_2 \|_{\H^0}.
$$
\end{enumerate}
Then $S(t)$ has an exponential attractor $\E$ contained in $\BB_{\H^1}(r_1)$.
\end{theorem}

\subsection{Proof of Theorem \ref{THM-ExpM}}
A first difficulty, in connection with the higher order estimate of point (v),
comes from the lack of compactness of the embedding $\H^1 \subset \H^0$,
for the embedding $\M^1 \subset \M^0$ is not compact either.
In order to bypass this obstacle, we introduce the more regular memory space
$$
\K = \Big\{ \eti \in \M^1 \,:\, \eti' \in \M^{0},\, \eti(0) = 0,\,\, \sup_{y \geq 1} y\T(y;\eti) < \infty \Big\},
$$
where
$$
\T(y;\eti) = \int_y^\infty \mu(s)\| \eti(s) \|^2_{Y^0}\d s, \quad y \geq 1,
$$
denotes the {\it tail function} of $\eti$.
Setting
$$
\HB(\eti) = \| \eti' \|^2_{\M^0} + \sup_{y \geq 1} y\T(y;\eti),
$$
$\K$ turns out to be a Banach space with the norm
$$
\| \eti \|^2_{\K} = \| \eti \|^2_{\M^1} + \HB(\eti),
$$
thus $\K$ is continuously embedded into $\M^1$.
Although $\K$ might not be reflexive, it can be proved that its closed balls
are closed in $\M^0$ (see \cite[Proposition 5.6]{GMPZ}).
Moreover, owing to \cite[Lemma 5.5]{PZ}, we have the compact embedding
$\K \Subset \M^0$.
Finally, we introduce the product space
$$
\W = X^1 \times \K\Subset \H^0.
$$
We will need a technical lemma, which easily follows (up to inessential modifications)
by collecting Lemma 3.3 and Lemma 3.4 in \cite{CPS}.

\begin{lemma}
\label{Lemma-CPS}
Given ${\rm T}\in(0,\infty]$, let $\eti_0 \in \M^0$ be such that
$\eti_0(0) = 0$ and
$\,\HB(\eti_0) < \infty$,
and let $\eti = \eti^t(s)$ be the solution to the Cauchy problem on $I_{\rm T}$ (where $I_{\rm T}=[0,{\rm T}]$ if
${\rm T}<\infty$ and $I_{\rm T}=[0,\infty)$ otherwise)
$$
\begin{cases}
\pt \eti^t = T\eti^t + \boldsymbol{w}(t), \\
\eti^{0} = \eti_0.
\end{cases}
$$
Assume that there exist $\q\in\D$ and $K\geq 0$ such that
$$
\| \boldsymbol{w}(t) \|^2_{Y^0} \leq \q(t)+K.
$$
Then, for all $t\in I_{\rm T}$, we have that $\eti^t(0)=0$ and
$$
\HB(\eti^t) \leq C_0 \e^{-\frac\delta2 t} + M K.
$$
Here, $C_0\geq 0$ depends on $\HB(\eti_0)$, $\q$ and $K$,
whereas $M>0$ is a universal constant and $\delta >0$ is given by \eqref{NEC}.
In particular, if  $\eti_0=0$ and $\q\equiv 0$, then $C_0=0$.
\end{lemma}

We now proceed to the proof by stating two more lemmas.

\begin{lemma}
\label{lem-ii-w}
There exists $R_2>0$ such that, given any $r\geq 0$,
$$
\| S(t)\z \|_\W \leq \q_r(t) + R_2
$$
for some $\q_r \in \D$ and every $\z \in \BB_\W(r)$.
\end{lemma}

\begin{proof}
Let $r\geq 0$, and let $\z = (\x_0, \eti_0)\in \BB_{\W}(r)$ be arbitrarily chosen.
The inclusion $\BB_{\W}(r) \subset \BB_{\H^1}(r)$ together with~(i) entail
$$
\| \x(t) \|_{X^1} \leq \q_r(t) + R_1
$$
for some $\q_r \in \D$. Owing to (iii), up to possibly changing $\q_r$ and $R_1$, the inequality
$$
\| \A\x(t) \|_{Y^0} \leq \q_r(t) + R_1
$$
is readily seen to hold.
Since $\eti^t$ solves the second equation of system \eqref{SYM}, exploiting Lemma~\ref{Lemma-CPS} with ${\rm T}=\infty$ we obtain
$$
\HB(\eti^t) \leq C_0 \e^{-\frac\delta2 t} + M R_1,
$$
where here $C_0$ depends on $r$. Recalling (i), this yields the desired conclusion.
\end{proof}

\begin{lemma}
\label{lem-expattr-w}
There is $r_2>0$ such that $\BB_\W(r_2)$ is exponentially attracting for $S(t)$.
\end{lemma}

\begin{proof}
It is enough to show that the ball $\BB_{\W}(r_2)$ exponentially attracts $\BB_{\H^1}(r_1)$.
Indeed, exploiting (ii) and the continuity~\eqref{contM}, we can apply the
transitivity property of exponential attraction \cite{FGMZ}, and conclude that
the basin of (exponential) attraction of $\BB_{\W}(r_2)$
coincides with the whole phase space.
Let then $(\x(t), \eti^t)$
be the solution with initial datum $\z = (\x_0, \eti_0) \in \BB_{\H^1}(r_1)$.
Along the proof, $\Q\in\I$ denotes a {\it generic} function.
Making use of (i) and (iii), we get
$$
\sup_{t \geq 0} \| S(t)\z \|_{\H^1} \leq \Q(r_1)\quad\Rightarrow\quad
\sup_{t \geq 0} \| \A\x(t)\|_{Y^0} \leq \Q(r_1).
$$
We make the decomposition
$$
S(t)\z = (0, \boldsymbol{\psi}^t)+(\x(t), \boldsymbol{\zeta}^t),
$$
where $\boldsymbol{\psi}^t$ and $\boldsymbol{\zeta}^t$ solve
$$
\begin{cases}
\pt \boldsymbol{\psi}^t = T\boldsymbol{\psi}^t, \\
\boldsymbol{\psi}^0 = \eti_0,
\end{cases}
\quad
\begin{cases}
\pt \boldsymbol{\zeta}^t = T\boldsymbol{\zeta}^t + \A\x(t), \\
\boldsymbol{\zeta}^0 = 0.
\end{cases}
$$
Condition \eqref{NEC} entails that $\boldsymbol{\psi}^t$ satisfies the uniform exponential decay
$$
\| \boldsymbol{\psi}^t \|_{\M^1}^2 \leq \Theta r_1^2\,\e^{-\delta t}.
$$
In particular, $\boldsymbol{\psi}^t$ is uniformly bounded in $\M^1$, the bound depending on $r_1$.
Since so is $\eti^t$, we learn that
$$
\sup_{t \geq 0} \| \boldsymbol{\zeta}^t \|_{\M^1} \leq \Q(r_1),
$$
which, combined with an application of Lemma \ref{Lemma-CPS} with ${\rm T}=\infty$, gives
$$
\| \boldsymbol{\zeta}^t \|^2_{\K} = \| \boldsymbol{\zeta}^t \|^2_{\M^1}
+ \HB(\boldsymbol{\zeta}^t) \leq \Q(r_1).
$$
We conclude that
$$
\| (\x(t),\boldsymbol{\zeta}^t) \|_\W \leq \Q(r_1),
$$
while $(0, \boldsymbol{\psi}^t)$ decays exponentially (in fact, even in $\H^1$).
Hence, the claim follows by choosing $r_2 = r_2(r_1)$ sufficiently large.
\end{proof}

At this point, with reference to Lemmas \ref{lem-ii-w} and~\ref{lem-expattr-w},
choose $\varrho > \max\{r_2, R_2\}$ and consider the ball
$$
\BB = \BB_{\W}(\varrho).
$$
Since $\varrho > r_2$, it is clear that
$\BB$ remains exponentially attracting for $S(t)$; whereas,
since $\varrho > R_2$,
we have for $t^*>0$ large enough
$$
\q_\varrho(t^*) + R_2 \leq \varrho,
$$
implying in turn
$$
S(t)\BB \subset \BB, \quad \forall t \geq t^*.
$$
Fixing a positive $\lambda<1$ and exploiting point (v), up to possibly increasing $t^*$, the map
$$
S = S(t^*): \BB \to \BB
$$
fulfills for every $\z_1, \z_2 \in \BB$ the decomposition
$$
S\z_1 - S\z_2 = L(\z_1, \z_2) + K(\z_1, \z_2),
$$
where
$$
\begin{aligned}
&\| L(\z_1, \z_2) \|_{\H^0} \leq \lambda \| \z_1 - \z_2 \|_{\H^0},\\
&\| K(\z_1, \z_2) \|_{\H^1} \leq C\| \z_1 - \z_2 \|_{\H^0}.
\end{aligned}
$$
Here and in what follows, $C>0$ is a {\it generic} constant depending on $\varrho$ and $t^*$.
Besides, applying Lemma~\ref{Lemma-CPS} with ${\rm T}=t^*$,
we get
$$
\HB(\bar\eti^{t^*})\leq C\| \z_1 - \z_2\|^2_{\H^0}.
$$
Hence,
$$
\| K(\z_1, \z_2) \|_{\W} \leq C\| \z_1 - \z_2\|_{\H^0}.
$$
Invoking a nowadays classical result\footnote{Actually, we are using a well-known generalization of the
result in \cite{CL,EMZ}, originally stated with the decomposition $S\z = L\z + K\z$.
In turn, $L(\z_1, \z_2)=L\z_1-L\z_2$ and $K(\z_1, \z_2)=K\z_1-K\z_2$.
Indeed, a closer look at the proofs of \cite{CL,EMZ} shows that it suffices to split the difference
$S\z_1 - S\z_2$ as in (v).}
(see \cite{CL,EMZ}),
there exists an exponential attractor $\E_{\rm d} \subset \BB$ for the discrete
semigroup $S^n$ made by the $n^{\rm th}$-iterations of $S$.
Then, we define the map $\S : [t^*,2t^*] \times \BB\to \BB$ by
$$
\S(t, \z) = S(t)\z,
$$
and we consider the set
$$
\E = \S([t^*,2t^*]\times \E_{\rm d}) \subset \BB.
$$
It is apparent that $\E$ is positively invariant for $S(t)$.
Besides, since $\eti^t$ is a solution to the second equation of system \eqref{SYM},
using (iii) and arguing exactly as in~\cite[Proposition 7.2]{GMPZ}, we get
$$
\| \eti^{t_1} - \eti^{t_2} \|_{\M^0}
\leq C(t_1 - t_2), \quad \forall \z \in \BB,
$$
for every $t_1 > t_2 \geq t^*$.
This estimate together with point~(iv) and \eqref{contM}
imply that the map $\S$ is H\"older continuous when $\BB$ is endowed with the $\H^0$-topology.
Therefore, $\E$ is compact in $\H^0$ and, due to the well-known properties of the fractal measure,
it has finite fractal dimension in $\H^0$.
By the continuity property~\eqref{contM}, we first deduce that $\E$
exponentially attracts $\BB$ under the action of the continuous semigroup $S(t)$,
and appealing once again to the transitivity property of exponential attraction~\cite{FGMZ},
we conclude that $\E$ is exponentially attracting on the whole space $\H^0$.
\qed

\smallskip
As a final comment, we point out that one interesting feature of the theorem
is that the higher regularity memory space $\K$ does not appear in the statement,
being completely hidden in the proof.
In particular, in concrete cases one has to work
only with the more natural spaces $\M^\imath$. This renders the result
more immediate and easier to apply.

\subsection{Further remarks}
We conclude the section with some remarks on the assumptions of Theorem \ref{THM-ExpM}.

\smallskip
\noindent
{\bf I.}
Hypothesis \eqref{NEC} on $\mu$ cannot be weakened since,
as proved in \cite{CHEP}, it is necessary (and sufficient
as well) for the uniform stability of the semigroup~$R(t)$ on $\M^0$.
Indeed, in order to have existence
of absorbing sets for systems with memory, postulated by
assumptions (i) and (ii), it is necessary that $R(t)$ be uniformly stable.
We also mention that, when $\Theta=1$,
hypothesis~\eqref{NEC} boils down the well-known condition devised by Dafermos~\cite{DAF}
\begin{equation}
\label{DAFFY}
\mu'(s) + \delta \mu(s) \leq 0.
\end{equation}
On the other hand, when $\Theta>1$, the gap between \eqref{NEC} and the above relation becomes quite effective.
Just note that, contrary to \eqref{NEC}, the latter condition
does not allow $\mu$ to have flat zones, or even horizontal inflection points.

\smallskip
\noindent
{\bf II.}
The continuity property~\eqref{contM} is not used in its full strength:
it is enough to require that formula~\eqref{contM} holds for every
$\z_1 \in \BB_{\H^0}(r)$ and $\z_2 \in \BB_{\H^1}(r_1)$.
Although in most cases such a distinction is inessential,
there are some situations where~\eqref{contM} as it is written is hard to prove.

\smallskip
\noindent
{\bf III.}
In fact, we obtain the boundedness of the exponential attractor $\E$ not
only in $\H^1$, but also in the compactly embedded space $\W \Subset \H^1$.
In particular, $\Pi_2\E$ belongs to $\dom(T)$.

\smallskip
\noindent
{\bf IV.}
A final comment on assumption (v). In certain cases,
it might be difficult to prove the estimate of the compact part~$K$ in $\H^1$.
Actually, it is possible to relax the condition by introducing an
intermediate space endowed with a norm weaker than $\H^1$, but still stronger than $\H^0$.
To this end, we need to assume the existence of two reflexive Banach spaces $\hat Y^1 $ and $\hat Y^{-1}$ such that
$$
Y^1 \subset \hat Y^1 \Subset Y^0 \subset \hat Y^{-1},
$$
for which the relation
$$
\|y\|_{Y^0}\leq c_0\|y\|_{\hat Y^{-1}}^\varpi\|y\|_{\hat Y^{1}}^{1-\varpi}
$$
holds for every $y \in Y^0$ and some $c_0>0$ and $\varpi\in(0,1]$.
Then, defining
$$
\hat \M^1 = L^2_{\mu}(\R^+;\hat Y^1),
$$
the conclusions of Theorem~\ref{THM-ExpM} remain true if
$K(t, \z_1, \z_2)$ is estimated in
$X^1 \times \hat \M^1$
and $\boldsymbol{w}(t)$ in $\hat Y^{-1}$, respectively.

\section{Exponential Attractors: The Minimal State Framework}
\label{SECEXPS}

\noindent
The abstract Theorem~\ref{THM-ExpM} can be given an equivalent formulation for
the semigroup $\hat S(t)$ acting on the extended state space $\V^0$. However,
once the existence of an exponential attractor $\E$ for $S(t)$ is attained,
the corresponding result for $\hat S(t)$ can be immediately deduced
under the following rather mild assumption.

\begin{assumption}
\label{asso}
For any initial datum $\hat\z\in \BB_{\V^0}(r)$, denote by $\x(t)=\Pi_1 \hat S(t)\hat\z$ the
first component of the solution
$\hat S(t)\hat\z$, and define
$$
\boldsymbol{\psi}^t(s)
=\int_0^{\min\{t,s\}}\A\x(t-y)\d y.
$$
Then the function $\boldsymbol{Z}(t)=(\x(t),\boldsymbol{\psi}^t)$ belongs to $\H^0$ for every $t\geq 0$
and
$$\sup_{t\geq 0}\|\boldsymbol{Z}(t)\|_{\H^0}\leq \Q(r)$$
for some $\Q\in\I$.
\end{assumption}

\begin{remark}
Making use of \eqref{qqqqqqqq}, we deduce the equality
\begin{equation}
\label{wwwwwwwwwww}
(\Lambda \boldsymbol{\psi}^t)(\tau)=\int_0^t\mu(\tau+s)\A\x(t-s)\d s.
\end{equation}
\end{remark}

\begin{theorem}
\label{THM-ExpS}
In addition to the general assumptions,
let Condition~\ref{asso} hold.
If the semigroup $S(t):\H^0\to\H^0$ possesses an exponential attractor $\E$, then the set
$$\hat\E=\LL\E$$
is an exponential attractor for the semigroup $\hat S(t):\V^0\to\V^0$.
\end{theorem}

\begin{remark}
If $\E$ is also bounded in $\H^1$, we readily infer from Lemma~\ref{TTTH}
the boundedness of $\hat\E$ in $\V^1$.
\end{remark}

\begin{proof}
By applying Lemma~\ref{lemme}, and exploiting the positive invariance of $\E$, we get at once
$$\hat S(t)\hat\E=\hat S(t)\LL \E=\LL S(t)\E \subset \LL \E=\hat\E.$$
Moreover, since $\LL$ is Lipschitz continuous, we infer that $\hat\E$ is compact and
$$\Frac_{\V^0}(\hat\E)=\Frac_{\V^0}(\LL\E)\leq \Frac_{\H^0}(\E)<\infty.
$$
We are left to prove the existence of $\hat\omega>0$ and $\Q\in\I$ for which
\begin{equation}
\label{theBigI}
\dist_{\V^0}(\hat S(t)\BB_{\V^0}(r),\hat\E)\leq \Q(r)\e^{-\hat\omega t}.
\end{equation}
To this end, let $r\geq 0$ be fixed, and let $\hat \z=(\x_0,\xei_0)\in\BB_{\V^0}(r)$.
Along this proof, $C>0$ will denote a {\it generic} constant depending (increasingly) only on $r$.
Setting then
$$\hat S(t)\hat \z=(\x(t),\xei^t),$$
we construct the function
$\boldsymbol{Z}(t)=(\x(t),\boldsymbol{\psi}^t)$ of Condition~\ref{asso},
which is bounded in $\H^0$ uniformly with respect to $t$, the bound depending on $r$.
Finally, we introduce the function
$$\hat{\boldsymbol{Z}} (t)=\LL \boldsymbol{Z}(t)=(\x(t),\Lambda \boldsymbol{\psi}^t).$$
The first step is showing that, for every $a,b\geq 0$,
\begin{equation}
\label{xxxxx1}
\|\hat S(a+b)\hat\z-\hat S(a)\hat{\boldsymbol{Z}}(b)\|_{\V^0} \leq C\e^{\ell a-\frac{\delta\kappa}2b}
\end{equation}
for some $\ell=\ell(r)\geq 0$ and $\kappa=\kappa(r)\in(0,1]$.
Indeed, from the continuous dependence estimate~\eqref{contS},
$$\|\hat S(a+b)\hat\z-\hat S(a)\hat{\boldsymbol{Z}}(b)\|_{\V^0} \leq C\e^{\ell a}
\|\hat S(b)\hat\z-\hat{\boldsymbol{Z}}(b)\|_{\V^0}^\kappa
=C\e^{\ell a}
\|\xei^b-\Lambda \boldsymbol{\psi}^b\|_{\S^0}^\kappa.
$$
On the other hand, exploiting the representation formula \eqref{REP} and \eqref{wwwwwwwwwww}, we obtain
$$\xei^b(\tau)-(\Lambda \boldsymbol{\psi}^b)(\tau)=
\xei_0(b+\tau).$$
Hence, by virtue of~\eqref{NECnu},
$$\|\xei^b-\Lambda \boldsymbol{\psi}^b\|_{\S^0}^2
=\int_b^\infty \nu(\tau-b)\|\xei_0(\tau)\|^2_{Y^0}\d \tau\leq \Theta\|\xei_0\|^2_{\S^0}\e^{-\delta b}
\leq C\e^{-\delta b}.
$$
This proves~\eqref{xxxxx1}.
Next, we show that, for every $a,b\geq 0$,
\begin{equation}
\label{xxxxx2}
\dist_{\V^0}(\hat S(a)\hat{\boldsymbol{Z}}(b),\hat\E)\leq C\e^{-\omega a}
\end{equation}
for some $\omega>0$. Indeed,
$$\dist_{\V^0}(\hat S(a)\hat{\boldsymbol{Z}}(b),\hat\E)
=\dist_{\V^0}(\LL S(a){\boldsymbol{Z}}(b),\LL\E)\leq \dist_{\H^0}(S(a){\boldsymbol{Z}}(b),\E),
$$
and since ${\boldsymbol{Z}}(b)$ is uniformly bounded in $\H^0$,
the desired conclusion follows from the fact that $\E$ is an exponential attractor for $S(t)$.
At this point, writing $t=a+b$ and making use of~\eqref{xxxxx1}-\eqref{xxxxx2}, we are led to
\begin{align*}
\dist_{\V^0}(\hat S(t)\hat\z,\hat\E)
&\leq \|\hat S(a+b)\hat\z-\hat S(a)\hat{\boldsymbol{Z}}(b)\|_{\V^0}
+\dist_{\V^0}(\hat S(a)\hat{\boldsymbol{Z}}(b),\hat\E)\\
&\leq C\big[\e^{\ell a-\frac{\delta\kappa}2b}+\e^{-\omega a}\big].
\end{align*}
Choosing
$$a=\varkappa t\qquad\text{with}\qquad \varkappa=\frac{\delta \kappa}{\delta\kappa+2\omega+2\ell},$$
the desired conclusion~\eqref{theBigI} is reached by setting $\hat\omega=\omega \varkappa$.
\end{proof}

Actually, in the proof above, the exponential rate $\hat\omega$ depends on $r$.
However, such a dependence can be removed by means of a simple argument.

\begin{lemma}
Assume that \eqref{theBigI} holds for $\hat\omega=\hat\omega(r)$. Then~\eqref{theBigI} holds for
some $\hat\omega$ independent of $r$ as well.
\end{lemma}

\begin{proof}
Since the set $\hat\E$ is bounded and attracting, it is clear that for some $r_0>0$ large
enough the set $\Bcal_0=\BB_{\V^0}(r_0)$ is absorbing for $\hat S(t)$.
In particular, for any fixed $r\geq 0$, there exists $t_r\geq 0$ such that
$$\hat S(t)\BB_{\V^0}(r)\subset \Bcal_0,\quad\forall t\geq t_r.$$
Hence, setting $\hat\omega=\hat\omega(r_0)$, we get
$$\dist_{\V^0}(\hat S(t)\BB_{\V^0}(r),\hat\E)\leq
\dist_{\V^0}(\hat S(t-t_r)\Bcal_0,\hat\E)\leq C\e^{-\hat\omega (t-t_r)},\quad \forall t\geq t_r,
$$
for some $C=C(r_0)>0$. On the other hand, it is readily inferred from~\eqref{contS} that
$$\dist_{\V^0}(\hat S(t)\BB_{\V^0}(r),\hat\E)\leq \Q(r),\quad\forall t<t_r,$$
for some $\Q\in\I$. Collecting the two inequalities above the claim follows.
\end{proof}

\section{Application to the Equation of Viscoelasticity}
\label{SecVISCO}

\noindent
We conclude our discussion
with an application of the abstract Theorems~\ref{THM-ExpM} and \ref{THM-ExpS} to
the damped wave equation with memory arising in the theory of isothermal viscoelasticity~\cite{RHN}.

\subsection{The model}
Let $\Omega \subset \R^3$ be a bounded domain with smooth boundary
$\partial\Omega$. For $t>0$, we consider the equation
\begin{equation}
\label{CONCEQ}
\ptt u - \Delta u - \int_{0}^\infty k(s) \Delta\pt u(t-s)\d s+ f(u) = g
\end{equation}
in the unknown $u = u(x, t): \Omega \times \R \to \R$, subject to the Dirichlet boundary condition
$$u(x, t)_{|x \in \partial \Omega} = 0.$$
The variable $u$ is assumed to be known for negative times $t\leq 0$, where it need not solve the equation.
Defining the kernel
$$h(s)=k(s)+1,$$
and performing an integration by parts, \eqref{CONCEQ}
takes the more familiar form
$$\ptt u - h(0)\Delta u - \int_{0}^\infty h'(s) \Delta u(t-s)\d s+ f(u) = g.$$
Besides the general assumptions of Sec.\ \ref{Sec1}, we further suppose that
the memory kernel $\mu$ has no jumps\footnote{The request that $\mu$ have no
jumps is actually made only to simplify the notation.
Indeed, in the presence of jumps, no significant changes are needed in the forthcoming proofs.}
(i.e.\ it is absolutely continuous on $\R^+$) and
\begin{equation}
\label{BELLO}
\mu'(s)<0,\quad\text{for a.e. }s>0.
\end{equation}
Concerning the other terms,
$g\in L^2(\Omega)$ is a time-independent external force, whereas
the nonlinearity $f \in \C^2(\R)$, with $f(0)=0$, satisfies the standard growth
and dissipation conditions
\begin{align}
\label{ass-f}
|f''(u)|&\leq c(1+|u|),\quad \\
\label{f2}
\liminf_{|u|\to \infty}  \frac{f(u)}u & > - \lambda_1,
\end{align}
where $\lambda_1>0$ is the first eigenvalue of the Dirichlet operator $-\Delta$.
Equation \eqref{CONCEQ} is a particular case of the more general family of damped wave equations with memory \cite{DWE1,DWE2,DWE3}
$$
\ptt u + \alpha\pt u - \beta\Delta\pt u - \Delta u - \int_{0}^\infty k(s) \Delta\pt u(t-s)\d s+ f(u) = g,
$$
where $\alpha,\beta \geq 0$. If either $\alpha>0$ or $\beta>0$, instantaneous damping terms are present.
Instead, the case $\alpha=\beta=0$ under consideration is much more challenging, since the whole
dissipation mechanism is contained in the convolution integral only, and it is substantially weaker.
It is also worth mentioning that \eqref{CONCEQ} can be viewed as a memory relaxation of the Kelvin-Voigt model
$$
\ptt u - \Delta u - \Delta \pt u + f(u) = g,
$$
the latter being recovered in the limiting situation when $k$ collapses into the Dirac mass at $0^+$.

\begin{remark}
\label{TDMVAC}
It is readily seen that \eqref{CONCEQ} can be written in the form \eqref{BASEL} by setting
$$
\x =
\begin{pmatrix}
u \\
\pt u
\end{pmatrix},
\quad
\A\x =
\begin{pmatrix}
0 \\
-\Delta \pt u
\end{pmatrix},
\quad
\B\x =
\begin{pmatrix}
-\pt u \\
-\Delta u + f(u) -g
\end{pmatrix}.
$$
\end{remark}

\subsection{Notation}
Let $A=-\Delta$ be the linear
strictly positive operator on $L^2(\Omega)$ with domain
$$
\dom(A)=H^2(\Omega)\cap H_0^1(\Omega).
$$
For $\sigma \in \R$, we define the compactly nested Hilbert spaces
$$
\HH^\sigma = \dom(A^{\sigma/2}),
$$
endowed with the inner products and norms
$$
\l \cdot,\cdot\r_{\sigma}
=\l A^{\sigma/2}\cdot,A^{\sigma/2}\cdot\r_{L^2(\Omega)},
\qquad\| \cdot\|_{\sigma}
=\|A^{\sigma/2}\cdot\|_{L^2(\Omega)}.
$$
The index $\sigma$ will be always omitted whenever zero. In particular,
$$
\HH^{-1}=H^{-1}(\Omega),\quad \HH=L^2(\Omega),\quad
\HH^1=H_0^1(\Omega), \quad
\HH^2=H^2(\Omega)\cap H_0^1(\Omega).
$$
Till the end of the paper, $\Q\in\I$ will denote a {\it generic} function.

\subsection{Translating the equation}
According to Sec.\ \ref{SECSOL} and Remark~\ref{TDMVAC}, upon defining properly the functional spaces, equation~\eqref{CONCEQ}
translates
into an ODE in the history space, as well as in the state space frameworks.
To this end, with reference
to Sec.\ \ref{SECGEN}, for $\imath = 0,1$, we set
\begin{align*}
X^\imath = \HH^{\imath+1} \times \HH^\imath, \qquad
Y^\imath = \{ 0 \} \times \HH^{\imath-1},
\end{align*}
and we define the spaces $\M^\imath,\S^\imath$ and $\H^\imath,\V^\imath$ accordingly.
Since the first component of $Y^\imath$ is degenerate,
abusing the notation we agree to identify $Y^\imath$ with its second
component $\HH^{\imath-1}$, as well as the spaces $\M^\imath,\S^\imath$
with their second components, to wit,
$$\M^\imath=L^2_\mu(\R^+;\HH^{\imath-1}),
\qquad
\S^\imath=L^2_\nu(\R^+;\HH^{\imath-1}).
$$
Similarly, setting
$$
\eti =
\begin{pmatrix}
0 \\
\eta
\end{pmatrix}
\qquad\text{and}\qquad
\xei =
\begin{pmatrix}
0 \\
\xi
\end{pmatrix},
$$
we will write $T\eta$ and $P\xi$ in place of $T\eti$ and $P\xei$, respectively.
Then, the concrete realizations of the abstract equations \eqref{SYM} and \eqref{SYS} read
\begin{equation}
\label{SYS-MEM}
\begin{cases}
\displaystyle
\ptt u + Au + \int_{0}^\infty \mu(s) \eta(s)\d s + f(u) = g, \\
\noalign{\vskip1mm}
\pt \eta = T\eta + A\pt u,
\end{cases}
\end{equation}
and
\begin{equation}
\label{SYS-STATE}
\begin{cases}
\displaystyle
\ptt u + Au + \int_{0}^\infty \xi(\tau)\d \tau + f(u) = g, \\
\noalign{\vskip1mm}
\pt \xi = P\xi + \mu A\pt u.
\end{cases}
\end{equation}
Both~\eqref{SYS-MEM} and~\eqref{SYS-STATE} generate strongly (in fact, jointly) continuous semigroups
$$S(t):\H^0\to\H^0
\qquad\text{and}\qquad
\hat S(t):\V^0\to\V^0.
$$
Moreover, the continuous dependence estimates~\eqref{contM}
and~\eqref{contS} hold for $\kappa = 1$ and, due to~\eqref{ass-f}-\eqref{f2}, for every $r\geq 0$  we have
\begin{equation}
\label{boundMS}
S(t)\BB_{\H^0}(r)\subset \BB_{\H^0}(\Q(r))
\qquad\text{and}\qquad
\hat S(t)\BB_{\V^0}(r)\subset \BB_{\V^0}(\Q(r)),
\end{equation}
uniformly as $t\geq 0$ (see \cite{comarpa,DWE1,Rocky}).
In particular, given initial data  $\z = (u_0, v_0, \eta_0)\in \H^0$ and
$\hat \z = (u_0, v_0, \xi_0) \in \V^0$, we have the corresponding solutions
$$S(t)\z = (u(t), \pt u(t), \eta^t)
\qquad\text{and}\qquad
\hat S(t)\hat \z = (u(t), \pt u(t), \xi^t ),$$
and the representation formulae~\eqref{REPETA} and~\eqref{REP} become
$$
\eta^t(s)=
\begin{cases}
Au(t)-Au(t-s) & \quad  s\leq t,\\
\eta_0(s-t)+Au(t)-A u_0 & \quad  s>t,
\end{cases}
$$
and
$$
\xi^t(\tau)=\xi_0(t+\tau)+\int_0^t\mu(\tau+s)A\pt u(t-s)\d s.
$$

\subsection{Exponential attractors}
We are now in the position to state the main result on the existence of
regular exponential attractors for the equation of viscoelasticity.
In fact, although~\eqref{CONCEQ} is perhaps the most important (and certainly the most studied) example
of equation with memory, the existence of an exponential attractor has never been proved before, not even for the
much simpler
model where $\mu$ satisfies the less general assumption~\eqref{DAFFY} and/or
an additional term of the form $\alpha\pt u$, accounting for dynamical friction, is present.
So far in the literature, only the issues of global attractors and convergence of single
trajectories have been addressed. In particular, the existence of the global attractor
and its regularity within the hypotheses of the present work has been established in~\cite{comarpa,Rocky}.
This somehow justifies the need
of an easy-to-handle theoretical tool allowing to treat this kind of equations.

\begin{theorem}
\label{MAIN}
The semigroup $S(t)$ on $\H^0$ generated by~\eqref{SYS-MEM} possesses an exponential attractor~$\E$ bounded in $\H^1$.
\end{theorem}

The proof of Theorem~\ref{MAIN} will be carried out in detail in the final Sec.\ \ref{SecLAST}.

\begin{remark}
It is worth mentioning that assumptions \eqref{NEC} and \eqref{BELLO} are still much weaker than the commonly used condition~\eqref{DAFFY}.
As a matter of fact, using the techniques devised in~\cite{PAT}, it is even possible to weaken~\eqref{BELLO}. Loosely speaking,
the conclusions of the theorem remain true if~\eqref{NEC} holds and $\mu$ is not ``too flat"; to make it precise, if
the {\it flatness rate} ${\mathbb F}$
of $\mu$ does not exceed $1/2$, where
$${\mathbb F}=\frac 1 {k(0)}\int_{\{s\,:\,\mu'(s)=0\}}
\mu(y)\d y.$$
In our case, we assumed ${\mathbb F}=0$.
\end{remark}

As a direct consequence of Theorem~\ref{MAIN}, we have

\begin{theorem}
\label{MAIN2}
The semigroup $\hat S(t)$ on $\V^0$ generated by~\eqref{SYS-STATE} possesses an exponential attractor~$\hat\E$ bounded in $\V^1$.
\end{theorem}

\begin{proof}
Owing to Theorem~\ref{MAIN}, we have to show that Condition~\ref{asso} is satisfied.
In which case, the desired conclusion follows from the abstract Theorem~\ref{THM-ExpS}.
To this aim, let
$$
\Pi_1 \hat S(t)\hat\z = (u(t), \pt u(t)),
$$
for an arbitrary initial datum $\hat\z = (u_0, v_0, \xi_0) \in \BB_{\V^0}(r)$.
Then, the second component (the first being trivially zero) $\psi^t$ of $\boldsymbol{\psi}^t$
fulfills the equality
$$
\psi^t(s) =
\begin{cases}
Au(t)-Au(t-s) & \quad  s\leq t,\\
Au(t)-A u_0 & \quad  s>t.
\end{cases}
$$
Exploiting \eqref{boundMS}, we readily obtain that
$$\|u(t)\|_1+\|\pt u(t)\|\leq \Q(r)
$$
and
$$
\| \psi^t(s)\|_{-1} = \| u(t) - u((t-s)_+)\|_{1}  \leq \Q(r),
$$
which clearly imply Condition~\ref{asso}.
\end{proof}

\section{Proof of Theorem \ref{MAIN}}
\label{SecLAST}

\noindent
All we need is checking the validity of conditions (i)-(v) of
Theorem~\ref{THM-ExpM}. To this end we set for $\sigma \in [0,1]$
$$
\M^\sigma = L^2_\mu(\R^+; \HH^{\sigma-1}),
$$
and we introduce the product spaces
$$
\H^\sigma = \HH^{\sigma+1} \times \HH^\sigma \times \M^\sigma.
$$
In the course of the investigation, we will borrow the following known results from~\cite{Rocky}.
\begin{itemize}
\item[$\bullet$] The semigroup $S(t)$ possesses a bounded absorbing set\footnote{Actually, the proof
of the existence of an absorbing set in \cite{Rocky} is indirect, being a consequence of the existence
of the global attractor, which is attained by means of gradient system techniques.} $\BB_0\subset\H^0$.
\smallskip
\item[$\bullet$] For every $\eps>0$ there exists $C_\eps>0$ such that
\begin{equation}
\label{lemma-int-ut}
\int_\tau^t \| \pt u(y) \| \d y \leq \eps(t - \tau) + C_{\eps}
\end{equation}
for every initial datum $\z \in \BB_0$. Here $\pt u$ denotes the second component of the solution.
\smallskip
\end{itemize}
Recall that $\BB_0 \subset \H^0$ is an absorbing set if
for every $r\geq 0$ there exists an entering time $t_r\geq 0$
such that
$$
S(t)\BB_{\H^0}(r) \subset \BB_0,
\quad \forall t \geq t_r.
$$
In what follows,
for any initial datum $\z = (u_0, v_0, \eta_0) \in \H^0$ and $t \geq 0$, we agree to call
$$
S(t)\z = (u(t), \pt u(t), \eta^t)
$$
the solution to system~\eqref{SYS-MEM} at time $t$.
The calculations that follow are understood to hold within a suitable regularization scheme;
in particular, the third component of the solution belongs to $\dom(T)$.
We will often use without explicit mention the Young,
the H\"older and the Poincar\'e inequalities, as well as the standard Sobolev embeddings, e.g.\
$\HH^1 \subset L^6(\Omega)$ and $\HH^2 \subset L^\infty(\Omega)$.

\subsection{Proof of point (i)}
We will actually prove a more general result.

\begin{proposition}
\label{abs-general}
For every $\sigma \in [0,1]$, there exists
$R_{\sigma}>0$ with the following property: given any $r\geq 0$
 there is $\q_{\sigma,r} \in \D$ such that
\begin{align*}
\|S(t)\z \|_{\H^{\sigma}}\leq \q_{\sigma,r}(t)+R_{\sigma}
\end{align*}
for all $\z \in
\BB_{\H^\sigma}(r)$.
\end{proposition}

The proof of the case $\sigma=0$ follows by combining~\eqref{boundMS} with the existence of the
absorbing set $\BB_0$. Instead, the case $\sigma>0$ requires some work. In order to avoid the presence of unnecessary
constants, we assume without loss of generality
$$k(0)=\int_0^{\infty}\mu(s)\d s=1.$$
Let us take for the moment
$$\z \in \BB_{\H^\sigma}(r)\cap\mathbb{B}_0.$$
Then we know from~\eqref{boundMS} that
$$\|S(t)\z \|_{\H^0} \leq C,$$
where, throughout this subsection, $C>0$ denotes
a {\it generic} constant, possibly depending on $\BB_0$. In turn,
by \eqref{ass-f},
\begin{equation}
\label{bddfg}
\|f(u(t))-g\|\leq C.
\end{equation}
We define the energy functional
$E_\sigma=E_\sigma(t)$ by
$$
E_\sigma=\|u\|^2_{\sigma+1}+\|\pt u\|^2_{\sigma}+
\|\eta\|^2_{\M^\sigma}+2\langle f(u)-g,A^\sigma u\rangle.
$$
Since by~\eqref{bddfg} and the fact that $\sigma\leq 1$
$$2|\langle f(u)-g,A^\sigma u\rangle|\leq 2\|f(u)-g\|\|u\|_{2\sigma}
\leq C\|u\|_{\sigma+1},$$
we readily obtain the controls
\begin{equation}
\label{trollo1}
\frac{1}{2}\|S(t)\z\|_{\H^{\sigma}}^2-C\leq E_\sigma(t)\leq
2\|S(t)\z\|_{\H^{\sigma}}^2+C.
\end{equation}

\begin{lemma}
\label{lemmaEsigma}
The functional $E_\sigma$ fulfills the differential inequality
\begin{equation}
\label{STIMA1}
\ddt E_\sigma-\int_0^{\infty}\mu'(s)\|\eta(s)\|^2_{\sigma-1}\d s
\leq C\|\pt u\|E_\sigma +C.
\end{equation}
\end{lemma}

\begin{proof}
Multiplying the first equation of \eqref{SYS-MEM} by $\pt u$ in $\HH^\sigma$ and the
second one by $\eta$ in $\M^{\sigma}$, we get
$$
\frac{\d}{\d t}E_\sigma=2\langle
T\eta,\eta\rangle_{\M^\sigma}+2\langle f'(u)\pt u,A^\sigma u
\rangle.
$$
The first term in the right-hand side satisfies the identity (see e.g.\ \cite{Rocky})
$$2\langle T\eta,\eta\rangle_{\M^\sigma}=\int_0^{\infty} \mu'(s)\|\eta(s)\|^2_{\sigma-1}\d s\leq 0,
$$
whereas by \eqref{ass-f} and the standard Sobolev (or Agmon if $\sigma=1$) inequalities we obtain
\begin{align*}
2\langle f'(u)\pt u,A^{\sigma}u \rangle
& \leq C\big(1+\|u\|^2_{L^{6/(1-\sigma)}}\big)\|\pt u\|\|A^{\sigma}u\|_{L^{6/(1+2\sigma)}}\\
& \leq C\big(1+\|u\|_1\|u\|_{\sigma+1}\big)\|\pt u\|\|u\|_{\sigma+1}\\
& \leq C\|\pt u\|\|u\|_{\sigma+1}^2+C.
\end{align*}
On account of \eqref{trollo1}, we are finished.
\end{proof}

To reach the desired conclusion, we have to improve \eqref{STIMA1}.
To this end, an additional functional is needed to reconstruct the energy.
First, in order to deal with the (possible) singularity of
$\mu(s)$ at zero, for any $\nu>0$ small we choose
$s_\nu >0$ such that
$$
\int_0^{s_\nu}\mu(s)\d s\leq \frac\nu 2,
$$
and we introduce the truncated kernel
$$\mu_\nu(s)=\mu(s_\nu)\chi_{(0,s_\nu]}(s)+\mu(s)\chi_{(s_\nu,\infty)}(s),
$$
where $\chi$ denotes the characteristic function. Besides, for any $\delta>0$ we set
\begin{align*}
{\mathcal P}_\delta[\eta] &=\int_{P_\delta}\mu(s)\|\eta(s)\|^2_{\sigma-1}\d s,\\
{\mathcal N}_\delta[\eta] &=\int_{N_\delta}\mu(s)\|\eta(s)\|^2_{\sigma-1}\d s,
\end{align*}
where
$$P_\delta=\big\{s\in\R^+:  \mu'(s)+\delta \mu(s)> 0\big\},
\qquad
N_\delta=\big\{s\in\R^+:  \mu'(s)+\delta \mu(s)\leq 0\big\}.$$
Note that
\begin{equation}
\label{ENNE}
{\mathcal N}_\delta[\eta]\leq -\frac1\delta \int_0^{\infty}\mu'(s)\|\eta(s)\|^2_{\sigma-1}\d s
\end{equation}
and
$${\mathcal P}_\delta[\eta]+{\mathcal N}_\delta[\eta]=\|\eta\|_{\M^{\sigma}}^2.$$
Finally, we define the functional $\Phi=\Phi(t)$ as
\begin{align*}
\Phi &=-\int_0^\infty \mu_\nu(s)\l \pt u,\eta(s)\r_{\sigma-1} \d s+(1-2\nu)\l \pt u,u\r_{\sigma}\\
&\quad +\int_{0}^\infty \bigg(\int_s^\infty \mu(y)\chi_{P_\delta}(y) \d y\bigg)\|\eta(s)-Au\|^2_{\sigma-1}\d s.
\end{align*}
Since, as shown in~\cite{Rocky}, condition \eqref{NEC} is equivalent to
$$k(s)\leq D \mu(s)$$
for some $D>0$, it is immediate to ascertain that
\begin{equation}
\label{FzCTRL}
|\Phi(t)|\leq C \|S(t)\z\|_{\H^{\sigma}}^2.
\end{equation}

Collecting Lemmas 5.3 and 5.4 in~\cite{Rocky}, and taking advantage of~\eqref{ENNE},
we obtain the following lemma,\footnote{Actually, Lemmas 5.3 and 5.4 in~\cite{Rocky} are proved for $\sigma=0$,
but the proofs are in fact the same for any $\sigma\in[0,1]$ (cf.\ \cite[Remark 5.5]{Rocky}).}
whose almost immediate proof is left to the reader.

\begin{lemma}\label{lemmaPhi}
For any $\nu,\delta>0$ small, there exists a positive constant $K=K(\nu,\delta)$
such that the following inequality holds:
\begin{align*}
& \ddt\Phi +\frac14\|u\|^2_{\sigma+1}+\nu\|\pt u\|^2_\sigma+\frac{1}{4}\|\eta\|_{\M^\sigma}^2\\
& \leq -K\int_0^{\infty}\mu'(s)\|\eta(s)\|^2_{\sigma-1}\d s+K\|f(u)-g\|^2.
\end{align*}
\end{lemma}

Accordingly, up to fixing $\nu,\delta>0$ small enough and recalling~\eqref{bddfg}, we conclude
that
\begin{equation}
\label{FIIIIII}
\ddt\Phi +\nu \|S(t)\z \|_{\H^{\sigma}}^2
\leq -C\int_0^{\infty}\mu'(s)\|\eta(s)\|^2_{\sigma-1}\d s+C.
\end{equation}

\begin{proof}[Conclusion of the proof of Proposition \ref{abs-general}]
At this point, for $\eps>0$ we introduce the further functional $\Gamma=\Gamma(t)$ as
$$\Gamma=E_\sigma+\eps\Phi.$$
By \eqref{trollo1} and \eqref{FzCTRL} it is apparent that
\begin{equation}
\label{S(t)z}
\frac{1}{4}\|S(t)\z\|_{\H^{\sigma}}^2-C\leq \Gamma(t)\leq
4\|S(t)\z\|_{\H^{\sigma}}^2+C
\end{equation}
for any $\eps>0$ small enough.
By collecting~\eqref{STIMA1} and~\eqref{FIIIIII} we obtain
$$
\ddt\Gamma +\nu\eps \|S(t)\z \|_{\H^{\sigma}}^2
\leq (1-C\eps)\int_0^{\infty}\mu'(s)\|\eta(s)\|^2_{\sigma-1}\d s +C\|\pt u\|E_\sigma+C.
$$
Choosing then $\eps$ sufficiently small, and recalling~\eqref{trollo1} and \eqref{S(t)z},
we end up with the differential inequality
$$\ddt\Gamma +2\gamma \Gamma
\leq C\|\pt u\|\Gamma+C
$$
for some $\gamma>0$.
Owing to the dissipation integral~\eqref{lemma-int-ut}, we can apply a modified version of the Gronwall lemma
(see e.g. Lemma 2.1 in \cite{DWE1}) to get
$$\Gamma(t)\leq C|\Gamma(0)|\e^{-\gamma t}+C.$$
A final use of \eqref{S(t)z} completes the proof when $\z \in \BB_{\H^\sigma}(r)\cap\mathbb{B}_0$.

As far as the general case is concerned, we observe that
there exists $t_r\geq 0$ such that $S(t_r)\BB_{\H^\sigma}(r)\subset\BB_0$.
It is then enough to show that
\begin{equation}
\label{second}
S(t)\BB_{\H^\sigma}(r)\subset \BB_{\H^\sigma}(\Q(r)),\quad\forall t\leq t_r.
\end{equation}
Indeed, once this is known,
we can write (for $t\geq t_r$)
$$S(t)\z=S(t-t_r)\boldsymbol{\tilde\z}
\qquad\text{with}\qquad
\boldsymbol{\tilde\z}=S(t_r)\z\in \BB_{\H^\sigma}(\Q(r))\cap\mathbb{B}_0,$$
and the claim follows by the previous step.
The proof of~\eqref{second} can be done by merely applying the Gronwall lemma to~\eqref{STIMA1}
on the time interval $[0,t_r]$,
the only difference being that now the constant $C$ will depend on $r$.
\end{proof}

\subsection{Proof of point (ii)}
Let $r\geq 0$ be given. According to~\cite{Rocky} (see Lemmas~4.3 and~4.4 therein),
for every $\z\in\BB_{\H^0}(r)$ the solution
$S(t)\z$ splits into the sum
$$
S(t)\z = L(t)\z + K(t)\z,
$$
where
$$
\|L(t)\z\|_{\H^0} \leq \Q(r)\e^{-\omega t}\qquad\text{and}\qquad
\|K(t)\z\|_{\H^{1/3}} \leq\Q(r).
$$
Here, $\omega>0$ is actually independent of $r$.
In particular, the (bounded) absorbing set $\BB_0$ is exponentially attracted by
a ball
$\BB_{1/3}$ of $\H^{1/3}$.
Besides, by Proposition~\ref{abs-general} with $\sigma=1/3$, we know that
$S(t)\BB_{1/3}$ remains uniformly bounded (with respect to $t$) in $\H^{1/3}$.
Hence, Lemma~4.5 in~\cite{Rocky} allows us to draw the uniform bound
$$\sup_{\z\in \BB_{1/3}}\|K(t)\z\|_{\H^1}\leq C$$
for some $C>0$ depending only on $\BB_{1/3}$.
This, together with the decay of $L(t)\z$,
tell that $\BB_{1/3}$ is exponentially attracted by
a ball
$\BB_{1}$ of $\H^{1}$.
By the transitivity property of exponential attraction \cite{FGMZ}, which applies since we have the
continuity~\eqref{contM},
we infer that $\BB_0$ is exponentially attracted by
$\BB_1$. At this point, since $\BB_0$ is absorbing and~\eqref{boundMS} holds,
it is a standard matter to conclude that $\BB_1$ exponentially attracts
every bounded subset of $\H^0$.
\qed

\subsection{Proof of point (iii)}
For $\x= (u, v)$, it is immediate to see that
$$\|\A \x\|_{Y^0}=\|A v\|_{-1}=\|v\|_1\leq \|\x\|_{X^1},$$
yielding the desired bound.
\qed

\subsection{Proof of point (iv)}
If $\z\in \BB_{\H^1}(r)$, by the previous point (i) we get
$$
\| S(t)\z \|_{\H^1} \leq \Q(r)\quad\Rightarrow\quad
\|\pt u(t)\|_1\leq \Q(r).
$$
Recalling~\eqref{ass-f}, we read from the first equation of~\eqref{SYS-MEM} the further bound
$$
\| \ptt u(t) \| \leq \Q(r).
$$
Therefore, (iv) holds with $p = \infty$.
\qed

\subsection{Proof of point (v)}
Let $r \geq 0$ be fixed. Along this proof, $C>0$ and $\Q\in\I$ will denote a {\it generic} constant
and function, respectively, both possibly depending on $r$.
For $\z_1,\z_2\in\BB_{\H^1}(r)$ arbitrarily chosen,
we deduce from point~(i) the bound
\begin{equation}
\label{bound-v}
\| S(t)\z_\imath \|_{\H^1}=\|(u_\imath(t), \pt u_\imath(t), \eta_\imath^t)\|_{\H^1} \leq \Q(r).
\end{equation}
We decompose the difference
$$
S(t)\z_1 - S(t)\z_2 =
(\bar u(t), \pt \bar u(t), \bar\eta^t)
$$
as follows:
$$
S(t)\z_1 - S(t)\z_2 = L(t, \z_1, \z_2) + K(t, \z_1, \z_2),
$$
where
$$L(t, \z_1, \z_2) = (v(t), \pt v(t), \xi^t)
\qquad\text{and}\qquad
K(t, \z_1, \z_2) = (w(t), \pt w(t), \psi^t)$$
fulfill the systems
$$
\begin{cases}
\displaystyle
\ptt  v + A  v + \int_0^\infty \mu(s)\xi(s)\d s = 0, \\
\pt \xi = T\xi + A\pt v,
\end{cases}
$$
and
$$
\begin{cases}
\displaystyle
\ptt  w + A  w + \int_0^\infty \mu(s)\psi(s)\d s + f(u_1) - f(u_2) = 0, \\
\pt \psi = T\psi + A\pt w,
\end{cases}
$$
with initial data
$$L(0, \z_1, \z_2)=\z_1 - \z_2\qquad\text{and}\qquad
K(0, \z_1, \z_2)=0.$$
As shown in \cite{PAT,Flat}, the first (linear) system
generates an exponentially stable contraction semigroup on $\H^0$. Thus,
\begin{equation}
\label{dec-lin-exp-dec}
\| L(t, \z_1, \z_2) \|_{\H^0} \leq B\e^{-\omega t}\| \z_1 - \z_2\|_{\H^0},
\end{equation}
for some $B\geq 1$ and $\omega>0$. Concerning the second system, we define the energy functionals
$$
E_K(t) = \| K(t, \z_1, \z_2) \|^2_{\H^1} =
\| w(t) \|^2_2 + \| \pt w(t)\|^2_1 + \| \psi^t\|^2_{\M^1}
$$
and
$$
\Lambda (t) =
E_K(t) +
2\l f(u_1(t)) - f(u_2(t)), w(t)\r_1.
$$
Owing to \eqref{ass-f} and \eqref{bound-v}, we draw
$$
2|\l f(u_1) - f(u_2), w\r_1|
\leq C\| \bar u\|_1 \|w\|_2
\leq  \frac12 \|w\|^2_2 + C\| \bar u\|^2_1
\leq \frac12 E_K + C\| \bar u\|^2_1.
$$
Since by \eqref{contM} (which holds for
$\kappa=1$)
$$\| \bar u(t)\|^2_1\leq \Q(t)\| \z_1 - \z_2 \|^2_{\H^0},$$
we conclude that
\begin{equation}
\label{Lambda-Energy}
\Lambda(t) \geq
\frac12 E_K(t) - \Q(t)\| \z_1 - \z_2\|^2_{\H^0}.
\end{equation}
Besides, by direct calculations,
\begin{align*}
\ddt \Lambda
&= 2\l f'(u_1)\pt \bar u, w\r_1 + 2\l (f'(u_1) - f'(u_2))\pt u_2, w\r_1 +
2\l T\psi,\psi\r_{\M^1} \\
&\leq 2\l f'(u_1)\pt \bar u, w\r_1 + 2\l (f'(u_1) - f'(u_2))\pt u_2, w\r_1.
\end{align*}
Appealing once again to \eqref{ass-f} and \eqref{bound-v},
$$
2\l f'(u_1)\pt \bar u, w\r_1
\leq C\|\pt \bar u\| \|w\|_2
$$
and
$$
2\l (f'(u_1) - f'(u_2))\pt u_2, w\r_1
\leq C\| \bar u \|_{L^6}\|\pt u_2\|_{L^3} \| w \|_2
\leq C\| \bar u \|_1\| w \|_2.
$$
Therefore, making use of \eqref{dec-lin-exp-dec} and \eqref{Lambda-Energy},
\begin{align*}
\ddt \Lambda(t)
&\leq C\| w(t) \|_2\big( \| \bar u(t) \|_1 + \|\pt \bar u(t)\|\big) \\
&\leq C\| w(t) \|_2\big(\| w(t) \|_1 + \|\pt w(t)\|
+\| v(t) \|_1 + \|\pt v(t)\|\big) \\
\noalign{\vskip 0.5mm}
&\leq C \big(E_K(t) + \sqrt{E_K(t)}\| \z_1 - \z_2 \|_{\H^0}\big) \\
\noalign{\vskip 1mm}
&\leq C\big(E_K(t) +\| \z_1 - \z_2 \|^2_{\H^0}\big)\\
\noalign{\vskip 1mm}
&\leq C\Lambda(t) + \Q(t)\| \z_1 - \z_2\|^2_{\H^0}.
\end{align*}
Observing that $\Lambda(0)=0$,
an application of the Gronwall Lemma yields
$$
\Lambda(t) \leq \Q(t)\| \z_1 - \z_2\|^2_{\H^0},
$$
and a further use of~\eqref{Lambda-Energy} leads to
$$
\| K(t, \z_1, \z_2) \|^2_{\H^1}=E_K(t) \leq \Q(t)\| \z_1 - \z_2\|^2_{\H^0}.
$$
Accordingly,
$$
\| A\pt w(t) \|_{-1} = \| \pt w(t) \|_1
\leq \sqrt{E_K(t)}
\leq \Q(t)\| \z_1 - \z_2\|_{\H^0}.
$$
Collecting \eqref{dec-lin-exp-dec} and the last two estimates, the claim follows.
\qed



\end{document}